\documentclass{article}
\usepackage[margin=1in]{geometry}% http://ctan.org/pkg/geometry

\AtEndDocument{\bigskip{\footnotesize%
    \textsc{$^{1}$Department of Applied Mathematics, Delhi Technological University, Delhi–110042, India} \par  
  \textit{E-mail address:} \texttt{suryagiri456@gmail.com} \par
  \addvspace{\medskipamount}
    \textsc{$^{*}$Department of Applied Mathematics, Delhi Technological University, Delhi–110042, India} \par
  \textit{E-mail address:} \texttt{spkumar@dtu.ac.in} \par
}}
\usepackage{graphicx}
\usepackage{hyperref}
\usepackage[caption = false]{subfig}
\usepackage{amsthm}
\usepackage{amssymb}
\usepackage{caption}
\usepackage{mathrsfs}
\usepackage{mathtools}
\usepackage{enumerate}
\usepackage{amssymb}
\usepackage{wrapfig}
\usepackage{float}
\allowdisplaybreaks

\usepackage{geometry}
\numberwithin{equation}{section}
\DeclareMathOperator{\RE}{Re}

\theoremstyle{plain}
\usepackage{lipsum}

\newtheorem{theorem}{Theorem}[section]
\newtheorem{corollary}[theorem]{Corollary}

\newtheorem{lemma}{Lemma}[section]

\theoremstyle{definition}

\theoremstyle{remark}
\newtheorem{remark}{Remark}[section]

\makeatother

\usepackage{authblk}
\begin{document}
\title{ Toeplitz determinants of Logarithmic coefficients for Starlike and Convex functions }
\author{Surya Giri$^{1}$ and S. Sivaprasad Kumar$^*$ }
%\address{Department of Applied Mathematics, Delhi Technological University, Delhi, SC 29208}
%\email{}
%\author{Surya Giri Mahant}
%\address{}
%\email{surya}
%\author[$\dagger$]{Surya Giri}
%\author[$\star$]{S. Sivaprasad Kumar}
%\author[$\ddag$]{J. Jones}

%%\affil[$\dagger$,$\star$]{Department of Applied Mathematics, Delhi Technological University, Delhi–110042, India}\\
%\vspace{0.5cm}
%\vspace{0.5cm}

  %\textit{E-mail address}, R.~Campbell: \texttt{campr@galois.psu.edu}
%\affil[$\star$]{Atmospheric Research Station,
%Pala Lundi, Fiji}
%\affil[$\ddag$]{Department of Philosophy, Freedman College,
%Periwinkle, Colorado 84320}
\date{}

%\title{Bohr-Rogosinski phenomenon for $\mathcal{S}^*(\psi)$ and $\mathcal{C}(\psi)$}
%	\thanks{K. Gangania thanks to University Grant Commission, New-Delhi, India for providing Junior Research Fellowship under }	

%	\author[Kamaljeet]{Kamaljeet Gangania}
%	\address{Department of Applied Mathematics, Delhi Technological University,
%%	\email{gangania.m1991@gmail.com}
	
%	\author[S. Sivaprasad Kumar]{S. Sivaprasad Kumar}
%	\address{Department of Applied Mathematics, Delhi Technological University,
%		Delhi--110042, India}
%	\email{spkumar@dce.ac.in}

\maketitle	
	
\begin{abstract} 
	In this study, we deal with the sharp bounds of certain Toeplitz determinants whose entries are the logarithmic coefficients of analytic univalent functions $f$ such that the quantity $z f'(z)/f(z)$ takes values in a specific domain lying in the right half plane. The established results provide the bounds for the classes of starlike and convex functions, as well as various of their subclasses.
\end{abstract}
\vspace{0.5cm}
\maketitle

\section{Introduction}
     Let $\mathcal{A}$ be the class of analytic functions  $f$ defined on the open unit disk $\mathbb{D}= \{ z\in \mathbb{C}: \vert z\vert < 1\}$ with the following Taylor series expansion:
\begin{equation}\label{first}
      f(z) = z+ \sum_{n=2}^\infty a_n z^n .
\end{equation}
     The  subclass of  $ \mathcal{A}$ consisting of all univalent functions is denoted by $\mathcal{S}$.
     Associated with each function $f \in \mathcal{S}$, consider
\begin{equation}\label{LogF}
       F_f(z) =\log \frac{f(z)}{z} = 2 \sum_{n=1}^{\infty} \gamma_n(f) z^n, \quad z \in \mathbb{D}, \quad \log{1}=0.
\end{equation}
     The number $\gamma_n:=\gamma_n(f)$, for each $n=1,2,3,\cdots$, is called  the logarithmic coefficients of $f$.
    Using the idea of logarithmic coefficients, Kayumov~\cite{Kayumov} proved the Brennan's conjecture for the conformal mappings.  Also, logarithmic coefficients play an important role in Milin's conjecture (\cite[p. 155]{Duren}, \cite{Milin}).
      Contrary to the coefficients of $f\in \mathcal{S}$, a little exact information is known about the coefficients of $\log( f(z)/z)$ when $f\in \mathcal{S}$. The Koebe function leads to the natural conjecture $\vert \gamma_n \vert \leq 1/n$, $n\geq 1$ for the class $\mathcal{S}$. However, this is false, even in order of magnitude (see~\cite[Section 8.1]{Duren}). For $f\in \mathcal{S}$, the only known bounds are
     $$  \vert\gamma_1 \vert \leq 1  \;\; \text{and} \;\; \vert \gamma_2 \vert \leq \frac{1}{2} + \frac{1}{e^2}.$$
     The problem of finding the estimates of $\vert \gamma_n\vert$ ($n\geq 3$) for the class $\mathcal{S}$ is still open.
     In past few years, various authors examined the bounds of $\vert \gamma_n \vert$ for functions in the subclasses of $\mathcal{S}$ instead of the whole class (see~\cite{Allu,Logr,ChoL,Thom2,Thom}) and the references cited therein).

     In geometric function theory, the classes of convex and starlike functions are the subclasses of $\mathcal{S}$ that have received the most attention.
    A function $f\in \mathcal{S}$ is said to be convex if $f(\mathbb{D})$ is convex. Let $\mathcal{C}$ denote the class of convex functions. It is well known that, $f\in \mathcal{C}$, if and only if
     $ \RE( ( 1 +  z f''(z))/f'(z)) >0 $ for $z \in \mathbb{D}.$
     A function $f\in \mathcal{S}$ is said to be starlike if $f(\mathbb{D})$ is starlike with respect to the origin. Let $\mathcal{S}^*$ denote the class of starlike functions. Analytically, $f\in \mathcal{S}^*$, if and only if
     $ \RE ({z f'(z)}/{f(z)} )>0$  for $z\in \mathbb{D}.$
     Let $\Omega$ be the class of all Schwarz functions and $\mathcal{P}$ denote  the class of analytic functions $p : \mathbb{D} \rightarrow \mathbb{C}$ such that $p(0)=1$ and $\RE p(z)>0$ for all $z\in \mathbb{D}$.
     An analytic function $f$ is said to be subordinate to the analytic function $g$, if there exists a Schwarz function $\omega$ such that $f(z)= g(\omega(z))$ for all $z \in \mathbb{D}$. It is denoted by $f \prec g$. Ma and Minda~\cite{MaMinda} unified various subclasses of starlike and convex functions. They defined
     $$ \mathcal{S}^*(\varphi) = \bigg\{ f \in \mathcal{S} : \frac{z f'(z)}{f(z)} \prec \varphi(z) \bigg\}$$
     and
     $$ \mathcal{C}(\varphi) = \bigg\{ f \in \mathcal{S} : 1+ \frac{z f''(z)}{f'(z)} \prec \varphi(z) \bigg\}, $$
     where $\varphi(z)$ is an analytic univalent functions with positive real part in $\mathbb{D}$, $\varphi(\mathbb{D})$ is symmetric with respect to the real axis starlike with respect to $\varphi(0)=1$, and $\varphi'(0)>0$.
     Let, for $z \in \mathbb{D}$, $\varphi$ has the series expansion
      $$ \varphi(z) = 1 + B_1 z + B_2 z^2 + B_3 z^3 + \cdots, \quad B_1 >0. $$
      Since $\varphi(\mathbb{D})$ is symmetric about the real axis and $\varphi(0)=1$, therefore all $B_i$'s are real. Further, $\varphi$ is a Carath\'{e}odory function, it follows that $ \vert B_n \vert \leq 2$, $n \in \mathbb{N}$~\cite[page-41]{Duren}.

      If we take $\varphi(z)=(1+ A z)/(1+ B z)$, $-1 \leq B < A \leq 1$, $\mathcal{S}^*(\varphi)$ and $\mathcal{C}(\varphi)$ reduce to the classes of Janowski starlike and convex functions, denoted by $\mathcal{S}^*[A,B]$ and $\mathcal{C}[A,B]$ respectively (see~\cite{Janow}). For $B=-1$ and $A=1-2 \alpha$, ($0 \leq \alpha <1$), the classes $\mathcal{S}^*(\alpha)= \mathcal{S}^*[1 - 2\alpha,-1]$ and $\mathcal{C}(\alpha)=\mathcal{C}[1 - 2\alpha,-1]$ are the well known classes of starlike and convex functions of order $\alpha$ $(0\leq \alpha <1)$ (see~\cite{Duren}).

     Toeplitz matrices and Toeplitz determinants arise in the field of pure as well as applied mathematics~\cite{OTplz}. They occur in  analysis, integral equations, image processing, signal processing, quantum mechanics and among other areas. For more applications, we refer to the survey article~\cite{LHLIM}. Toeplitz matrices contain constant entries along their diagonals. For $f(z)= z+ \sum_{n=2}^\infty a_n z^n \in \mathcal{A}$, the Toeplitz determinant is given by
\begin{equation}\label{intil}
     T_{m,n}(f)= \begin{vmatrix}
	a_n & a_{n+1} & \cdots & a_{n+m-1} \\
	{a}_{n+1} & a_n & \cdots & a_{n+m-2}\\
	\vdots & \vdots & \vdots & \vdots\\
    {a}_{n+m-1} & {a}_{n+m-2} & \cdots & a_n\\
	\end{vmatrix},
\end{equation}
    where $m, n \in \mathbb{N}$. In case of the class $\mathcal{S}^*$ and $\mathcal{C}$, the bound of $\vert T_{2,n}(f) \vert$, $\vert T_{3,1}(f)\vert$ and $\vert T_{3,2}(f)\vert$ were examined by Ali et al.~\cite{AliW} in 2017.
    Motivated by this work, for small values of $m$ and $n$, various authors studied the bounds of $\vert T_{m,n}(f)\vert$ for various subclasses of $\mathcal{S}$  in past few years~\cite{Ahuja,Cudna,Giri,Lecko,Obra}.

    Hankel and Toeplitz matrices are closely related to each other. Hankel matrices contain constant entries along the reverse diagonals. Ye and Lim \cite{LHLIM} showed that any $n \times n$ matrix over $\mathbb{C}$ generically can be written as the product of some Toeplitz matrices or Hankel matrices.
   Recently, Kowalczyk and Lecko~\cite{Kowal} introduced the Hankel determinant whose entries were the logarithmic coefficients of functions in $\mathcal{A}$. They studied the sharp estimates of second order Hankel determinant of logarithmic coefficients for functions belonging to  $\mathcal{S}^*$ and $\mathcal{C}$, which is further generalized for the classes $\mathcal{S}^*(\alpha)$ and $\mathcal{C}(\alpha)$ by the same authors in~\cite{Kowal2}.
   Also, Mundalia and Kumar~\cite{Mundalia} studied the same problem for the certain subclasses of close-to-convex functions.

  Motivated by these works and considering the significance of Toeplitz determinant and logarithmic coefficients, we define
\begin{equation}\label{TplzLog}
     T_{m,n}(\gamma_f)= \begin{vmatrix}
	\gamma_n & \gamma_{n+1} & \cdots & \gamma_{n+m-1} \\
	{\gamma}_{n+1} & \gamma_n & \cdots & \gamma_{n+m-2}\\
	\vdots & \vdots & \vdots & \vdots\\
    {\gamma}_{n+m-1} & {\gamma}_{n+m-2} & \cdots & \gamma_n\\
	\end{vmatrix}.
\end{equation}
    Consequently, we obtain
    $$  T_{2,1}(\gamma_f) = \gamma_1^2 -\gamma_2^2 \quad \text{and} \quad T_{2,2}(\gamma_f)  =  \gamma_2^2 -\gamma_3^2.  $$
    A comparison of same powers of $z$ in (\ref{LogF}) yields that
\begin{equation}\label{gammas}
    \gamma_1 = \frac{a_2}{2},\quad \gamma_2 = \frac{1}{4}(2 a_3 - a_2^2)\; \text{and} \; \gamma_3 = \frac{1}{2} \bigg( a_4 - a_2 a_3 + \frac{1}{3} a_2^3 \bigg).
\end{equation}
    In this paper, we derive the sharp estimates of $ \vert T_{2,1}(\gamma_f) \vert$, $ \vert T_{2,2}(\gamma_f) \vert$ and $\vert T_{3,2}(f)\vert$ for the classes $\mathcal{S}^*(\varphi)$ and $\mathcal{C}(\varphi)$. The established bounds lead to a number of new and already known results for different subclasses of starlike and convex functions when $\varphi$ is appropriately chosen.

    The following lemmas are required to prove the main results.
    %%%%%%%%%%%%%%%%%%%%%%%%%%%%%%%%%%%%%%%%%%%%%%%%%%%%%%%%%%%%%%%%%%%%%%%%%%%%%%%%%%%%%%%%%%%%%%%%%%%%%%%%%%%%%%%%%
    %%%%%%%%%%%%%%%%%%%%%%%%%%%%%%%%%%%%%%%%%%%%%%%%%%%%%%%%%%%%%%%%%%%%%%%%%%%%%%%%%%%%%%%%%%%%%%%%%%%%%%%%%%%%%%%%%
\begin{lemma}\cite{Prokh}\label{Old}
    If $\omega(z) = \sum_{n=1}^\infty c_n z^n \in \Omega$ and $(\mu , \nu) \in \cup_{i=1}^3 D_i$, then
    $$ \vert c_3 + \mu c_1 c_2 + \nu c_1^3 \vert \leq \vert \nu \vert ,$$
    where
\begin{align*}
     D_1 &= \bigg\{ (\mu, \nu) : \vert \mu \vert \leq 2, \; \nu \geq 1 \bigg\}, \;\;  D_2 = \bigg\{ (\mu, \nu) : 2 \leq \vert \mu \vert \leq 4, \; \nu \geq \frac{1}{12} (\mu^2 + 8) \bigg\},
\end{align*}
   and
   $$  D_3 = \bigg\{ (\mu, \nu) : \vert \mu \vert \geq 4, \; \nu \geq \frac{2}{3} (\vert \mu\vert - 1) \bigg\}. $$
\end{lemma}
\begin{lemma}\cite[Theorem 1]{Efraim}\label{lemmaa}
    Let $p(z) = 1+ \sum_{n=1}^\infty p_n z^n \in \mathcal{P}$ and $\mu \in \mathbb{C}$. Then
    $$ \vert p_n - \mu p_k p_{n-k} \vert \leq 2 \max \{ 1, \vert 2 \mu -1 \vert \}, \quad 1 \leq k \leq n-1 . $$
    The inequality is sharp for the function $p(z) = (1+z)/(1-z)$ or its rotation when $\vert 2 \mu -1 \vert \geq 1$. In case of $\vert 2 \mu -1 \vert < 1$, the inequality is sharp for $p(z) =(1+z^n)/(1-z^n)$ or its rotations.
\end{lemma}
 %%%%%%%%%%%%%%%%%%%%%%%%%%%%%%%%%%%%%%%%%%%%%%%%%%%%%%%%%%%%%%%%%%%%%%%%%%%%%%%%%%%%%%%%%%%%%%%%%%%%%%%%%%%%%%%%%%%%%%%%%%%%%%%%%%%%
    %%%%%%%%%%%%%%%%%%%%%%%%%%%%%%%%%%%%%%%%%%%%%%%%%%%%%%%%%%%%%%%%%%%%%%%%%%%%%%%%%%%%%%%%%%%%%%%%%%%%%%%%%%%%%%%%%%%%%%%%%%%%%%%%%%%%
\section{Main results}
    We begin with the bounds of $\vert T_{2,1}(\gamma_f) \vert$ and $\vert T_{2,2}(\gamma_f) \vert$ for the classes $\mathcal{S}^*(\varphi)$  and $\mathcal{C}(\varphi)$.
\begin{theorem}\label{thm1}
    Let $\varphi(z) = 1+ B_1 z +B_2 z^2 +B_3 z^3 +\cdots $ and $f\in \mathcal{S}^*(\varphi)$. If $ \vert B_2 \vert \geq B_1$, then
   $$  \vert \gamma_1^2 - \gamma_2^2 \vert \leq \frac{B_1^2}{4} + \frac{B_2^2}{16}. $$
    The estimate is sharp.
\end{theorem}
\begin{proof}
     Let $f \in \mathcal{S}^*(\varphi)$ be of the form (\ref{first}). Then there exists a Schwarz function, say $\omega(z) = \sum_{n=1}^\infty c_n z^n $ such that
\begin{equation}\label{phiw}
    \frac{z f'(z)}{f(z)} = \varphi(\omega(z)) , \quad  z \in \mathbb{D}.
\end{equation}
    From the Taylor series expansions of $f$ and $\varphi$, we obtain
\begin{equation}\label{phiw1}
      \frac{z f'(z)}{f(z)} = 1 + a_2 z + (-a_2^2 + 2 a_3) z^2 + (a_2^3 - 3 a_2 a_3 + 3 a_4) z^3 + \cdots
\end{equation}
     and
\begin{equation}\label{phiw2}
    \varphi(\omega(z)) =  1 + B_1 c_1 z + (B_2 c_1^2 + B_1 c_2) z^2 + (B_3 c_1^3 + 2 B_2 c_1 c_2 + B_1 c_3) z^3 + \cdots.
\end{equation}
    By comparing the same powers in (\ref{phiw}) using (\ref{phiw1}) and (\ref{phiw2}), coefficients $a_2$, $a_3$ and $a_4$ can be expressed as
\begin{equation}\label{a2a3}
    a_2 = B_1 c_1, \;\;  a_3 = \frac{1}{2} (B_1^2 c_1^2 + B_2 c_1^2 + B_1 c_2 )
\end{equation}
   and
\begin{equation}\label{a4}
    a_4 =  \frac{1}{48} ( (8 B_1^3 + 24 B_1 B_2  + 16 B_3 )c_1^3 + ( 24 B_1^2  + 32 B_2) c_1 c_2 + 16 B_1 c_3).
\end{equation}
    Further, applying $\vert c_n \vert \leq 1$, we get
\begin{equation}\label{a2}
    \vert a_2 \vert \leq B_1.
\end{equation}
    Ali et al. \cite[Theorem 1]{AliFS} established the bound of Fekete-Szeg\"{o} functional for $p-$valent functions, which for $p=1$ gives
\begin{equation*}\label{FSSPHI}
\begin{aligned}
\vert a_3 - \lambda a_2^2 \vert \leq
\left\{
\begin{array}{ll}
    \dfrac{1}{2}( B_1^2 + B_2 - 2 \lambda B_1^2 ), & \text{if} \;\; 2 \lambda B_1^2 \leq B_1^2 + B_2 - B_1; \\ \\
    \dfrac{1}{2}B_1,                      & \text{if} \;\;  B_1^2 + B_2 - B_1 \leq 2 \lambda B_1^2    \leq  B_1^2 + B_2 + B_1 ; \\ \\
    \dfrac{1}{2} ( - B_1^2 - B_2 + 2 \lambda B_1^2 ),  & \text{if} \;\;  2 \lambda B_1^2 \geq B_1^2 + B_2 + B_1.
\end{array}
\right.
\end{aligned}
\end{equation*}
   Since $\vert B_2\vert  \geq B_1$, hence the above inequality directly yields
\begin{equation}\label{FSBS}
    \vert a_3 - \frac{1}{2} a_2^2 \vert \leq \frac{\vert B_2\vert}{2}.
\end{equation}
   From (\ref{gammas}), we obtain
\begin{equation}\label{relat}
    \vert \gamma_1^2 - \gamma_{2}^2 \vert =  \bigg\vert \frac{1}{4} \left( a_2^2 - \left( a_3 - \frac{a_2^2}{2} \right)^2 \right) \bigg\vert
                              \leq \frac{1}{4} \bigg( \vert a_2 \vert^2 + \bigg\vert a_3 - \frac{a_2^2}{2} \bigg\vert^2 \bigg).
\end{equation}
   The required bound follows from (\ref{relat}) by using the bounds of $\vert a_2 \vert $ and $\vert a_3 -  (a_2^2)/2 \vert$ from (\ref{a2}) and (\ref{FSBS}) respectively.

   To show the sharpness of the bound, consider the analytic function $k_\varphi: \mathbb{D}\rightarrow \mathbb{C}$ given by
\begin{equation}\label{extS}
    k_\varphi(z) =  z \exp \int_0^z \frac{\varphi(i t)-1}{t} dt = z + i B_1 z^2 - \frac{1}{2}(B_1^2 + B_2 ) z^3+ \cdots.
\end{equation}
   Clearly, $k_\varphi \in \mathcal{S}^*(\varphi)$ and for this function, a simple computation gives
   $$ \vert \gamma_1^2 - \gamma_2^2 \vert =  \frac{4 B_1^2 + B_2^2}{16} ,$$
   which shows that the bound is sharp.
\end{proof}
\begin{theorem}\label{thm2}
    Let $\varphi(z) = 1+ B_1 z +B_2 z^2 +B_3 z^3 +\cdots $ and $f\in \mathcal{C}(\varphi)$. If $ \vert  B_2 + \frac{1}{4} B_1^2 \vert \geq B_1 $, then
\begin{equation}\label{finalb}
    \vert \gamma_1^2 - \gamma_2^2 \vert \leq \frac{B_1^2}{16} + \frac{1}{144} \bigg(B_2 + \frac{B_1^2}{4} \bigg)^2.
\end{equation}
    The estimate is sharp.
\end{theorem}
\begin{proof}
    Suppose $f \in \mathcal{C}(\varphi)$ be of the form (\ref{first}). Then there exists a Schwarz function $\omega(z)= \sum_{n=1}^\infty c_n z^n$ such that
\begin{equation*}\label{Kphiw}
     1 + \frac{ z f''(z)}{f'(z) } = \varphi(\omega(z)), \quad  z \in \mathbb{D}.
\end{equation*}
    After comparing the coefficients of identical powers of $z$ with the Taylor series expansion of $f$, $\varphi$ and $\omega$ in the above equation, the coefficients $a_2$ and $a_3$ can be expressed as
\begin{equation}\label{a2a3Kk}
    a_2 = \frac{B_1 c_1}{2}, \quad  a_3 = \frac{1}{6} (B_1^2 c_1^2 + B_2 c_1^2 + B_1 c_2)
\end{equation}
  and
\begin{equation}\label{a4K}
      a_4=  \frac{1}{12} ( (4 B_1^3  + 3 B_1 B_2  + B_3 ) c_1^3 + ( 3 B_1^2  + 2 B_2 ) c_1 c_2 + B_1 c_3).
\end{equation}
  Applying the bound $\vert c_n \vert \leq 1$, we obtain
\begin{equation}\label{a2K}
    \vert a_2 \vert \leq \frac{B_1}{2}.
\end{equation}
   For $f\in \mathcal{C}(\varphi)$, Ma and Minda~\cite[Theorem 3]{MaMinda} established the following bound
\begin{equation*}\label{FSSPHI}
\begin{aligned}
\vert a_3 - \lambda a_2^2 \vert \leq
\left\{
\begin{array}{ll}
    \dfrac{1}{6}( B_2 -\frac{3}{2} \lambda B_1^2 + B_1^2  ), & \text{if} \;\; 3 \lambda B_1^2 \leq 2 (B_1^2 + B_2 - B_1); \\ \\
    \dfrac{1}{6}B_1,                      & \text{if} \;\;  2 (B_1^2 + B_2 - B_1) \leq 3 \lambda B_1^2    \leq  2( B_1^2 + B_2 + B_1) ; \\ \\
    \dfrac{1}{6} (- B_2  + \frac{3}{2} \lambda B_1^2  - B_1^2 ),  & \text{if} \;\;  2( B_1^2 + B_2 + B_1) \leq 3 \lambda B_1^2.
\end{array}
\right.
\end{aligned}
\end{equation*}
   Since $ \vert  B_2 + \frac{1}{4} B_1^2 \vert \geq B_1 $ holds, the above inequality directly gives
\begin{equation}\label{a3a2K}
    \vert a_3 - \frac{1}{2} a_2^2 \vert \leq  \frac{1}{6}\vert B_2 + \frac{1}{4} B_1^2 \vert.
\end{equation}
   Using the bounds of $\vert a_2 \vert$ and $\vert a_3 -(a_2^2)/2 \vert $  for $f \in \mathcal{C}(\varphi)$ given in (\ref{a2K}) and (\ref{a3a2K}), respectively, we obtain
   $$  \vert \gamma_1^2 - \gamma_{2}^2 \vert \leq \frac{1}{4} \bigg( \vert a_2 \vert^2 + \bigg\vert a_3 - \frac{a_2^2}{2} \bigg\vert^2 \bigg) \leq \frac{B_1^2}{16} + \frac{1}{144} \bigg(B_2 + \frac{B_1^2}{4} \bigg)^2. $$
   The equality case in (\ref{finalb}) holds for the function $h_\varphi$ given by
\begin{equation}\label{exth}
   1 + \frac{z h_\varphi''(z)}{h_\varphi'(z)} = \varphi(i z).
\end{equation}
   Clearly, $h_\varphi \in \mathcal{C}(\varphi)$ and for this function, we have
   $$ \gamma_1 = \frac{i B_1}{4} \;\; \text{and}\;\;  \gamma_2 = - \frac{1}{12} ( B_2 + \frac{B_1^2}{4}),  $$
   which shows that the bound in (\ref{finalb}) is sharp.
\end{proof}
%%%%%%%%%%%%%%%%%%%%%%%%%%%%%%%%%%%%%%%%%%%%%%%%%%%%%%%%%%%%%%%%%%%%%%%%%%%%%%%%%%%%%%%%%%%%%%%%%%%%%%%%%%%%%%%%%%%%%%%%%%%%%%%%%%%%%%%%%%%%%%%%%%%%%%%%%%%%%%%%%%%%%
     %%%%%%%%%%%%%%%%%%%%%%%%%%%%%%%%%%%%%%%%%%%%%%%%%%%%%%%%%%%%%%%%%%%%%%%%%%%%%%%%%%%%%%%%%%%%%%%%%%%%%%%%%%%%%%%%%%%%%%%%%%%%%%%%%%%%%%%%%%%%%%%%%%%%%%%%%%%%%%%%%%%%%
\begin{theorem}\label{thm3}
    Let $\varphi(z) = 1 + B_1 z+ B_2 z^2 +B_3 z^3 +\cdots$ and $f \in \mathcal{S}^*(\varphi)$. If $ \vert B_2 \vert \geq B_1$ and $ (\mu_1 , \nu_1 ) \in \cup_{i=1}^{3} D_i$ hold, then
    $$ \vert \gamma_2^2 - \gamma_3^2  \vert \leq \frac{1}{144}(9 B_2^2 + 4 B_3^2), $$
   where $\mu_1 = 2 B_2/B_1$ and $\nu_1 = B_3/B_1$.
    The bound is sharp.
\end{theorem}
\begin{proof}
    Suppose $f\in \mathcal{S}^*(\varphi)$ be of the form (\ref{first}). Then from (\ref{gammas}), we have
\begin{equation}\label{g2g3}
\begin{aligned}
     \vert \gamma_2^2 - \gamma_3^2 \vert &= \frac{1}{4} \bigg\vert \bigg( a_3 - \frac{a_2^2}{2} \bigg)^2  - \bigg( \frac{a_2^3}{3} - a_2 a_3 + a_4 \bigg)^2 \bigg\vert \\
                          &\leq \frac{1}{4} \bigg( \bigg\vert a_3 - \frac{a_2^2}{2} \bigg\vert^2  + \bigg\vert \frac{a_2^3}{3} - a_2 a_3 + a_4 \bigg\vert^2 \bigg).
\end{aligned}
\end{equation}
    From (\ref{a2a3}) and (\ref{a4}) for $f\in \mathcal{S}^*(\varphi)$, using the values of $a_2$, $a_3$ and $a_4$  ,we obtain
    %From (\ref{a2a3}) and (\ref{a4}), we have
\begin{equation*}
    \bigg\vert \frac{a_2^3}{3} - a_2 a_3 + a_4  \bigg\vert = \frac{B_1}{3} \vert c_3 + \mu_1 c_1 c_2 + \nu_1 c_1^3 \vert,
\end{equation*}
   where $\mu_1 = 2 B_2/B_1$ and $\nu_1 = B_3/B_1$. Since $\vert B_2 \vert \geq B_1$ holds, therefore $(\mu_1, \nu_1)$ is a member of either $D_1$, $D_2$ or $D_3$. Thus, from Lemma~\ref{Old}, we get
\begin{equation}\label{c3c1c2}
    \bigg\vert \frac{a_2^3}{3} - a_2 a_3 + a_4  \bigg\vert \leq \frac{\vert B_3 \vert}{3}.
\end{equation}
   Using the bounds from (\ref{FSBS}) and (\ref{c3c1c2})  in the inequality (\ref{g2g3}), the required bound is obtained.

   The sharpness of the bound can be seen by the function $k_\varphi$ given by (\ref{extS}). As for this function, we have
   $ \gamma_2 = -B_2/4 $, $\gamma_3 = -i B_3/6$ and
   $$  \gamma_2^2 - \gamma_3^2 = \frac{1}{144} (9 B_2^2 + 4 B_3^2), $$
   which proves the sharpness.
\end{proof}
%%%%%%%%%%%%%%%%%%%%%%%%%%%%%%%%%%%%%%%%%%%%%%%%%%%%%%%%%%%%%%%%%%%%%%%%%%%%%%%%%%%%%%%%%%%%%%%%%%%%%%%%%%%%%%%%%%%%%%%%%%%%%%%%%%%%%%%%%%%%%%%%%%%%%%%%%%%%%%%%%%%%%
     %%%%%%%%%%%%%%%%%%%%%%%%%%%%%%%%%%%%%%%%%%%%%%%%%%%%%%%%%%%%%%%%%%%%%%%%%%%%%%%%%%%%%%%%%%%%%%%%%%%%%%%%%%%%%%%%%%%%%%%%%%%%%%%%%%%%%%%%%%%%%%%%%%%%%%%%%%%%%%%%%%%%%
\begin{theorem}\label{thm4}
   Let $\varphi(z) = 1 + B_1 z+ B_2 z^2 +B_3 z^3 +\cdots$ and $ f \in \mathcal{C}(\varphi)$. If $ \vert  B_2 + \frac{1}{4} B_1^2 \vert \geq B_1 $ and $(\mu_2, \nu_2)\in \cup_{i=1}^{3} D_i$ holds, then
   $$  \vert \gamma_2^2 - \gamma_3^2 \vert \leq  \frac{B_1^4 + 8 B_1^2 B_2 + 16 B_2^2 + B_1^2 B_2^2 + 4 B_1 B_2 B_3 + 4 B_3^2}{2304}, $$
   where  $\mu_2=  (B_1^2 + 4 B_2 )/( 2 B_1) $ and $\nu_2 = (B_1 B_2 + 2 B_3 )/(2 B_1)$. The bound is sharp.
\end{theorem}
\begin{proof}
     In view of the equations (\ref{a2a3Kk}) and (\ref{a4K}) for $f(z)= z+ \sum_{n=2}^\infty a_n z^n \in \mathcal{C}(\varphi)$, we have
    $$ \bigg\vert \frac{a_2^3}{3} - a_2 a_3 + a_4   \bigg\vert =  \frac{B_1}{12}  \bigg\vert c_3  + \mu_2  c_1 c_2 + \nu_2 c_1^3  \bigg\vert. $$
    As by the hypothesis $ \vert  B_2 + \frac{1}{4} B_1^2 \vert \geq B_1 $ holds, therefore $(\mu_2, \nu_2 )$ belongs to either $ D_1$, $ D_2$ or $ D_3$. Hence, from Lemma~\ref{Old}, we obtain
\begin{equation}\label{c3c1c2K}
   \bigg\vert \frac{a_2^3}{3} - a_2 a_3 + a_4   \bigg\vert \leq \frac{\vert B_1 B_2 + 2 B_3 \vert}{24}.
\end{equation}
    Applying the bound from (\ref{a3a2K}) and (\ref{c3c1c2K}) in the inequality (\ref{g2g3}), we get
    $$  \vert \gamma_2^2 -\gamma_3^2 \vert \leq  \frac{B_1^4 + 8 B_1^2 B_2 + 16 B_2^2 + B_1^2 B_2^2 + 4 B_1 B_2 B_3 + 4 B_3^2}{2304}. $$
    It is a simple exercise to check that
     the equality case holds for the function $h_\varphi \in \mathcal{C}(\varphi)$ given by (\ref{exth}).
\end{proof}
  %%%%%%%%%%%%%%%%%%%%%%%%%%%%%%%%%%%%%%%%%%%%%%%%%%%%%%%%%%%%%%%%%%%%%%%%%%%%%%%%%%%%%%%%%%%%%%%%%%%%%%%%%%%%%%%%%%%%%%%%%%%%%%%%%%%%
     %%%%%%%%%%%%%%%%%%%%%%%%%%%%%%%%%%%%%%%%%%%%%%%%%%%%%%%%%%%%%%%%%%%%%%%%%%%%%%%%%%%%%%%%%%%%%%%%%%%%%%%%%%%%%%%%%%%%%%%%%%%%%%%%%%%%%%%%%%%%%
\subsection{Some Special Cases}
    Since the classes $\mathcal{S}^*(\varphi)$ and $\mathcal{C}(\varphi)$ generalize various subclasses of starlike  and convex functions, therefore,
    for the appropriate choice of $\varphi$, whenever the Taylor series coefficients of $\varphi$ satisfy the conditions in Theorem~\ref{thm1}-\ref{thm4}, we obtain the sharp bounds of $\vert T_{2,1}(\gamma_f)\vert$ and $\vert T_{2,2}(\gamma_f)\vert$ for the corresponding class.

   In case of $\varphi(z)=(1+ A z)/(1+ B z)$ $(-1 \leq B < A \leq 1)$, we have $\mathcal{S}^*[A,B]=\mathcal{S}^*((1+ A z)/(1+ B z))$ and $\mathcal{C}[A,B]=\mathcal{C}((1+ A z)/(1+ B z))$. The series expansion of $(1+ A z)/(1+ B z)$ shows that
   $B_1= (A-B)$, $B_2 = B^2 -A B $ and $B_3 = A B^2 - B^3$.  Thus, Theorem~\ref{thm1}-\ref{thm4}  lead us to the following:
\begin{corollary}\label{crl1}
    Let $f\in \mathcal{S}^*[A,B]$  be of the form (\ref{first}), where $-1 \leq B < A \leq 1$.
\begin{enumerate}[(i)]
  \item If $\vert  B^2 -A B \vert \geq  A - B$, then
   $$ \vert \gamma_1^2 - \gamma_2^2  \vert \leq  \frac{ (A - B)^2 (4 + B^2)}{16} .$$
   \item If $\vert  B^2 -A B \vert \geq  A - B$, and $(\mu_1, \nu_1)\in \cup_{i=1}^3 D_i $, then
    $$ \vert \gamma_2^2 - \gamma_3^2  \vert \leq \frac{(A - B)^2 B^2 (4 B^2 + 9)}{144}, $$
    where $\mu_1 =-2 B$ and $\nu_1 = B^2.$
\end{enumerate}
\end{corollary}
\begin{corollary}\label{crl2}
     Let $f\in \mathcal{C}[A,B]$  be of the form (\ref{first}), where $-1 \leq B < A \leq 1$.
\begin{enumerate}[(i)]
  \item  If $ \vert A^2 - 6 A B + 5 B^2 \vert \geq 4( A - B)$, then
   $$ \vert \gamma_1^2 - \gamma_2^2  \vert \leq \frac{(A - B)^2 ( A^2 + 25 B^2 - 10 A B  + 144 )}{2304}. $$
  \item If $ \vert A^2 - 6 A B + 5 B^2 \vert \geq 4( A - B)$ and $(\mu_2, \nu_2)\in \cup_{i=1}^3 D_i $, then
   $$ \vert \gamma_2^2 - \gamma_3^2  \vert \leq  \frac{(A - B)^2 (A^2 ( B^2 + 1) + B^2 ( 9 B^2 + 25 )  - 2 A B ( 3 B^2 + 5 ) )}{2304}, $$
   where $\mu_2 = (A - 5 B)/2 $ and $\nu_2 = (B ( 3 B -A ))/2.$
\end{enumerate}
\end{corollary}

   By taking $A= 1- 2 \alpha$, $0 \leq \alpha <1$ and $B=-1$, the following results follow from Corollary~\ref{crl1} and Corollary~\ref{crl2}.
\begin{corollary}\label{crl6}
   If $f \in \mathcal{S}^*(\alpha)$, $0 \leq \alpha <1$, then
  $$ \vert \gamma_1^2 - \gamma_2^2  \vert \leq \frac{5}{16} (2 - 2 \alpha)^2 \;\; \text{and}\;\; \vert \gamma_2^2 - \gamma_3^2  \vert \leq \frac{13}{144} (2 - 2 \alpha)^2 .$$
\end{corollary}
\begin{corollary}\label{crl7}
    If $f \in \mathcal{C}(\alpha)$, $0 \leq \alpha <1$, then
    $$ \vert \gamma_1^2 - \gamma_2^2  \vert \leq \frac{( \alpha -1 )^2 ( \alpha^2 - 6 \alpha  +45)}{144} \;\;\text{and}\;\;
    \vert \gamma_2^2 - \gamma_3^2  \vert \leq \frac{( \alpha -1 )^2 ( 2 \alpha^2  - 10 \alpha  + 13 ) }{144}.   $$
\end{corollary}
   In particular, for $\alpha=0$, Corollary~\ref{crl6} and Corollary~\ref{crl7} give the bounds for the classes $\mathcal{S}^*$ and $\mathcal{C}$ respectively.
\begin{corollary}
   If $f \in \mathcal{S}^*$, then
  $$ \vert \gamma_1^2 - \gamma_2^2  \vert \leq \frac{5}{4} \;\; \text{and} \;\; \vert \gamma_2^2 - \gamma_3^2  \vert \leq  \frac{13}{36} .$$
\end{corollary}
\begin{corollary}
   If $f \in \mathcal{C}$, then
    $$ \vert \gamma_1^2 - \gamma_2^2  \vert \leq \frac{5}{16} \;\; \text{and} \;\;  \quad \vert \gamma_2^2 - \gamma_3^2  \vert \leq  \frac{13}{144} .$$
\end{corollary}
%%%%%%%%%%%%%%%%%%%%%%%%%%%%%%%%%%%%%%%%%%%%%%%%%%%%%%%%%%%%%%%%%%%%%%%%%%%%%%%%%%%%%%%%%%%%%%%%%%%%%%%%%%%%%%%%%%%%%%%%%%%%%%%%%%%%
     %%%%%%%%%%%%%%%%%%%%%%%%%%%%%%%%%%%%%%%%%%%%%%%%%%%%%%%%%%%%%%%%%%%%%%%%%%%%%%%%%%%%%%%%%%%%%%%%%%%%%%%%%%%%%%%%%%%%%%%%%%%%%%%%%%%%%%%%%%%%%
\section{Bounds of $\vert \det T_{3,2}(f)\vert $}
    Ali et al. \cite[Theorem 1]{AliFS} derived the sharp estimates of Fekete-Szeg\"{o} functional for $p-$valent functions belonging to $\mathcal{S}^*(\varphi)$, which for $p=1$ immediately gives the following estimates of $\vert a_4 \vert$.
\begin{lemma}\cite[Theorem 1]{AliFS}\label{lemma3}
    Let $\varphi(z) = 1+ B_1 z +B_2 z^2 +B_3 z^3+ \cdots$, and
    $$q_1= \frac{3 B_1^2 + 4 B_2 }{2 B_1 }, \quad q_2= \frac{B_1^3 + 3 B_1 B_2 + 2 B_3 }{2 B_1 }.$$
    If  $f \in \mathcal{S}^*(\varphi)$ is of the form (\ref{first}) such that $(q_1, q_2) \in \cup_{i=1}^3 D_i$, then
    $$  \vert a_4 \vert \leq  \frac{ B_1^3 + 3 B_1 B_2 + 2 B_3  }{6  }.$$
    The bound is sharp.
\end{lemma}
\begin{theorem}\label{thm5}
   Let $\varphi(z) = 1+ B_1 z +B_2 z^2 +B_3 z^3 +\cdots$ such that
\begin{equation*}\label{c1}
     6 B_1^2 \leq B_1 (3 B_1^2 + 2 B_2) \leq B_1^2 + 2 B_1^4 + 3 B_1^2 B_2 + 3 B_2^2 - 2 B_1 B_3,
\end{equation*}
   and
    $$ q_1= \frac{3 B_1^2 + 4 B_2 }{2 B_1 } , \quad q_2= \frac{B_1^3 + 3 B_1 B_2 + 2 B_3 }{2 B_1 }.$$
    If $f \in \mathcal{S}^*(\varphi)$ and  $(q_1, q_2) \in \cup_{i=1}^3 D_i$, then
    $$ \vert T_{3,2}(f)\vert \leq  \bigg(B_1 + \frac{ B_1^3 + 3 B_1 B_2 + 2 B_3 }{6  } \bigg) \bigg(  B_1^2 + \frac{B_1^4}{3} + \frac{B_1^2 B_2}{2} + \frac{B_2^2}{2} - \frac{B_1 B_3}{3}  \bigg).   $$
    The bound is sharp.
\end{theorem}
\begin{proof}
    Let $f \in \mathcal{S}^*(\varphi)$ be of the form (\ref{first}). Then from (\ref{phiw}), we have
    $$ z f'(z) = f(z) \varphi(\omega(z)), \quad z\in \mathbb{D}.  $$
    Corresponding to the Schwarz function $\omega$, there exists $p(z) = 1 + \sum_{n=1}^\infty p_n z^n \in \mathcal{P}$ such that $w(z) = (p(z)-1)/(p(z)+1)$. The comparison of identical powers of $z$ using the power series expansions of $f$, $\varphi$ and $p$ yield
\begin{align*}
    a_2 &=  \frac{B_1 p_1}{2},\; a_3 = \frac{1}{8} ( B_1^2 - B_1 + B_2 ) p_1^2 + 2 B_1 p_2)
\end{align*}
   and
   $$ a_4 = \frac{1}{48} \bigg( ( B_1^3 - 3 B_1^2 + 2 B_1 - 4 B_2 + 3 B_1 B_2 + 2 B_3 ) p_1^3 + ( 6 B_1^2 - 8 B_1 + 8 B_2) p_1 p_2 + 8 B_1 p_3 \bigg). $$
   Using these values of $a_2$, $a_3$ and $a_4$ in terms of $p_1$, $p_2$ and $p_3$, it follows that
\begin{align*}
    \vert a_2^2 - 2 a_3^2 + a_2 a_4  \vert  =  \bigg\vert  \frac{B_1^2 p_1^2}{4} &- \frac{(B_1^2 - 3 B_1^3 + 2 B_1^4 - 2 B_1 B_2 + 3 B_1^2 B_2 + 3 B_2^2 - 2 B_1 B_3) p_1^4 }{96} \\
   & - \frac{B_1 ( 3 B_1^2 -2 B_1  + 2 B_2) p_1^2 p_2}{48} - \frac{B_1^2 }{8} p_2^2 + \frac{B_1^2 }{12} p_1 p_3 \bigg\vert.
\end{align*}
   Keeping in mind that $B_1^2 + 2 B_1^4 + 3 B_1^2 B_2 + 3 B_2^2 - 2 B_1 B_3 - B_1 (3 B_1^2 + 2 B_2) \geq  0$ and by applying the bound $\vert p_n \vert \leq 2$, $n \in \mathbb{N}$ (see~\cite[Page- 41]{Duren}), we get
\begin{align*}
    \vert a_2^2 - 2 a_3^2 + a_2 a_4  \vert  \leq \frac{3 B_1^2}{2}  &+ \frac{(B_1^2 - 3 B_1^3 + 2 B_1^4 - 2 B_1 B_2 + 3 B_1^2 B_2 + 3 B_2^2 - 2 B_1 B_3)  }{6}  \\
   &  + \frac{B_1^2 }{6} \bigg\vert p_3 - \bigg( \frac{3 B_1^2 -2 B_1  + 2 B_2}{4 B_1} \bigg) p_1 p_2 \bigg\vert.
\end{align*}
   Since $3 B_1^2 + 6 B_2 \geq 6 B_1$, therefore from Lemma \ref{lemmaa}, we obtain
\begin{equation}\label{a2,a3,a4}
   \vert a_2^2 - 2 a_3^2 + a_2 a_4  \vert \leq B_1^2 + \frac{B_1^4}{3} + \frac{B_1^2 B_2}{2} + \frac{B_2^2}{2} - \frac{B_1 B_3}{3}.
\end{equation}
     Further, we have $\vert a_2 - a_4 \vert \leq \vert a_2 \vert + \vert a_4 \vert$. Using the bounds of $\vert a_2 \vert $ and $\vert a_4\vert$ from (\ref{a2}) and Lemma \ref{lemma3} respectively, we get
    $$ \vert a_2 - a_4 \vert \leq B_1 + \frac{ B_1^3 + 3 B_1 B_2 + 2 B_3  }{6  }.$$
     From (\ref{intil}), a simple computation reveals that
\begin{equation}\label{a2,a4}
    \vert T_{3,2}(f)\vert = \vert (a_2 - a_4) (a_2^2 - 2 a_3^2 + a_2 a_4 ) \vert.
\end{equation}
   The required estimated is determined by putting the bounds given in (\ref{a2,a3,a4}) and $(\ref{a2,a4})$ in the above equation.

   The function $k_\varphi$ defined by (\ref{extS}) plays the role of extremal functions. As for this function, we have
   $$  a_2 = i B_1, \quad a_3 =- \frac{1}{2}(B_1^2 + B_2) , \quad a_4 = - \frac{i}{6} (B_1^3 + 3 B_1 B_2 + 2 B_3)$$
   and
\begin{align*}
     \vert T_{3,2}(k_\phi) \vert = \bigg(B_1 + \frac{ B_1^3 + 3 B_1 B_2 + 2 B_3  }{6  } \bigg) \bigg(  B_1^2 + \frac{B_1^4}{3} + \frac{B_1^2 B_2}{2} + \frac{B_2^2}{2} - \frac{B_1 B_3}{3}  \bigg)
\end{align*}
  proving the sharpness.
\end{proof}
%%%%%%%%%%%%%%%%%%%%%%%%%%%%%%%%%%%%%%%%%%%%%%%%%%%%%%%%%%%%%%%%%%%%%%%%%%%%%%%%%%%%%%%%%%%%%%%%%%%%%%%%%%%%%%%%%%%%%%%%%%
%%%%%%%%%%%%%%%%%%%%%%%%%%%%%%%%%%%%%%%%%%%%%%%%%%%%%%%%%%%%%%%%%%%%%%%%%%%%%%%%%%%%%%%%%%%%%%%%%%%%%%%%%%%%%%%%%%%%%%%%%%
\begin{theorem}\label{thm6}
   Let $\varphi(z) = 1 + B_1 z + B_2 z^2 + B_3 z^3 +\cdots$ such that
\begin{equation}\label{26}
    16 B_1^2 - 4 B_1 B_2 \leq 7 B_1^3 \leq 5 B_1^4 + 2 B_1^2 - 4 B_1 B_2 + 7 B_1^2 B_2 + 8 B_2^2 - 6 B_1 B_3 ,
\end{equation}
    and
   $$  q_1 = \frac{3 B_1^2 + 4 B_2}{2 B_1}, \quad q_2= \frac{B_1^3 + 3 B_1 B_2 + 2 B_3 }{2 B_1}. $$
    If $f \in \mathcal{C}(\varphi)$ and $(q_1, q_2) \in \cup_{i=1}^{3} D_i $, then
   $$ \vert T_{3,2}(f)  \vert \leq \frac{1}{144} \bigg( \frac{B_1}{2} + \frac{B_1^3 + 3 B_1 B_2 + 2 B_3}{24} \bigg) ( 5 B_1^4 + 36 B_1^2  + 7 B_1^2 B_2 + 8 B_2^2 - 6 B_1 B_3). $$
   The bound is sharp.
\end{theorem}
\begin{proof}
    Suppose $f \in \mathcal{C}(\varphi)$ be of the form (\ref{first}), then we have
    $$ f'(z) + z f''(z) = f'(z) \varphi(\omega(z)) . $$
    Corresponding to the Schwarz function $\omega(z)= \sum_{n=1}^\infty c_n z^n$, there exists $p(z) = 1+ \sum_{n=1}^\infty p_n z^n \in \mathcal{P}$ such that $w(z)= (p(z)-1)/(p(z)+1)$. The comparison of same powers of $z$ in the above equation after the series expansions yield that
\begin{align*}
     a_2 = \frac{B_1 p_1}{4} , \quad  a_3 = \frac{1}{24} (( B_1^2 -B_1 + B_2) p_1^2 + 2 B_1 p_2)
\end{align*}
   and
\begin{equation}\label{a4p}
    a_4 = \frac{1}{192} \bigg(( B_1^3 - 3 B_1^2  + 2 B_1  - 4 B_2 + 3 B_1 B_2 + 2 B_3) p_1^3 + (
      6 B_1^2 + 8 B_2 -8 B_1 ) p_1 p_2 + 8 B_1 p_3 \bigg).
\end{equation}
   Using these expressions for $a_2$, $a_3$ and $a_4$ in terms of the coefficients $p_1$, $p_2$ and $p_3$, a simple computation gives
\begin{align*}
     \vert a_2^2 - 2 a_3^2   & + a_2 a_4 \vert
    % & = \bigg\vert \frac{1}{2304} \bigg( (2 B_1^2 - 7 B_1^3 + 5 B_1^4 - 4 B_1 B_2 + 7 B_1^2 B_2 + 8 B_2^2 - 6 B_1 B_3 ) p_1^4 \\
    %   &+ 32 B_1^2 p_2^2 -144 B_1^2 p_1^2 + ( 14 B_1^3 -8 B_1^2 + 8 B_1 B_2 ) p_1^2 p_2 - 24 B_1^2 p_1 p_3 \bigg) \bigg\vert \\
     = \bigg\vert \frac{1}{2304} \bigg( (2 B_1^2 - 7 B_1^3 + 5 B_1^4 - 4 B_1 B_2 + 7 B_1^2 B_2 + 8 B_2^2 - 6 B_1 B_3 ) p_1^4 \\
   &+ 32 B_1^2 p_2^2 -144 B_1^2 p_1^2 - 24 B_1^2 p_1 \bigg( p_3 - \frac{( 14 B_1^3 -8 B_1^2 + 8 B_1 B_2 )}{24 B_1^2 } p_1 p_2  \bigg) \bigg) \bigg\vert .
\end{align*}
   In view of the hypothesis $2 B_1^2  + 5 B_1^4 - 4 B_1 B_2 + 7 B_1^2 B_2 + 8 B_2^2 - 6 B_1 B_3 \geq 7 B_1^3$ and by applying the bound $\vert p_n \vert \leq 2$ ($n \in \mathbb{N})$, we get
\begin{align*}
     \vert a_2^2 - 2 a_3^2  + a_2 a_4 \vert & \leq  \frac{1}{2304} \bigg( 16 (2 B_1^2 - 7 B_1^3 + 5 B_1^4 - 4 B_1 B_2 + 7 B_1^2 B_2 + 8 B_2^2 - 6 B_1 B_3 ) \\
    & + 128 B_1^2  + 576 B_1^2  + 48 B_1^2  \bigg( \bigg\vert p_3 - \frac{( 14 B_1^3 -8 B_1^2 + 8 B_1 B_2 )}{24 B_1^2 } p_1 p_2  \bigg\vert \bigg) \bigg) .
\end{align*}
    Since $ 7 B_1^2  + 4 B_2 \geq 16 B_1 $ holds, therefore from Lemma \ref{lemma3}, it follows that
\begin{equation}\label{a2a3a4k}
     \vert a_2^2 - 2 a_3^2 + a_2 a_4 \vert \leq \frac{1}{144}( 36 B_1^2 + 5 B_1^4 + 7 B_1^2 B_2 + 8 B_2^2 - 6 B_1 B_3).
\end{equation}
    Now, we only need to maximize $\vert a_2 - a_4 \vert$ for $f\in \mathcal{C}(\varphi)$.
     By the one to one correspondence between the class $\mathcal{P}$ and the class of Schwarz functions, the coefficients $a_4$ in (\ref{a4p}) can be expressed as
     $$ a_4 = \frac{1}{12} B_1 ( c_3  +  q_1 c_1 c_2 +  q_2 c_1^3  ) ,$$
     where $q_1 = (3 B_1^2 + 4 B_2) /(2 B_1)$ and $q_2 = (B_1^3 + 3 B_1 B_2 + 2 B_3) /(2 B_1)$. As by the hypothesis $(q_1, q_2) \in \cup_{i=1}^3 D_i$, from Lemma \ref{Old}, we obtain
\begin{equation}\label{a4kk}
      \vert a_4 \vert \leq  \frac{B_1^3 + 3 B_1 B_2 + 2 B_3}{24 }.
\end{equation}
     Employing the bounds of $\vert a_2 \vert$ and $\vert a_4 \vert$ from (\ref{a2K}) and (\ref{a4kk}) respectively, we get
\begin{equation}\label{a2a4K}
     \vert a_2 - a_4 \vert \leq \vert a_2 \vert +\vert a_4\vert  \leq \frac{B_1}{2} +  \frac{B_1^3 + 3 B_1 B_2 + 2 B_3}{24 }.
\end{equation}
    Thus, applying the bounds of   $\vert a_2^2 - 2 a_3^2  + a_2 a_4 \vert$ and $\vert a_2 - a_4 \vert$  from (\ref{a2a3a4k}) and (\ref{a2a4K})  respectively in (\ref{a2,a4}), we get the desired result.

    The result is sharp for the function $h_\varphi$ defined in (\ref{exth}). As for this function, we have
    $ a_2 = i B_1/2 ,$ $a_3 = - (B_1^2 + B_2)/6 ,$  $a_4 = - i (B_1^3 + 3 B_1 B_2 + 2 B_3)/24 $ and
    $$ \vert T_{3,2}(f) \vert=  \frac{1}{144} \bigg( \frac{B_1}{2} + \frac{B_1^3 + 3 B_1 B_2 + 2 B_3}{24} \bigg) ( 5 B_1^4 + 36 B_1^2  + 7 B_1^2 B_2 + 8 B_2^2 - 6 B_1 B_3) $$
    proving the sharpness of the bound.
\end{proof}
\subsection{Special Cases}
    For the classes $\mathcal{S}^*[A,B]$ and $C[A,B]$, we have $\varphi(z)=(1+ A z)/(1+ B z)$ and the series expansion gives $B_1= A-B$, $B_2 = B^2 -A B$ and $B_3=A B^2 - B^3$. Hence, we deduce the following results immediately from Theorem \ref{thm5} and Theorem \ref{thm6}.
\begin{corollary}\label{crl4}
    For $-1  \leq B < A \leq 1$, let
    $$ 6 (A - B)^2 \leq (3 A - 5 B) (A - B)^2  \leq  (A - B)^2 ( 2 A^2 - 7 A B + 6 B^2 +1),  $$
    and
    $$ q_1= \frac{3 A - 7 B}{2}, \quad q_2= \frac{A^2 - 5 A B + 6 B^2}{2}.$$
    If $f \in \mathcal{S}^*[A,B]$ and  $(q_1, q_2) \in \cup_{i=1}^3 D_i$, then
    $$ \vert T_{3,2}(f) \vert \leq \frac{1}{36} (A - B)^2 ( 2 A^2 - 7 A B + 6 B^2 +6  ) (  A^3 +6 A - 6 B -6 A^2 B + 11 A B^2 - 6 B^3)  .$$
    The estimates is sharp.
\end{corollary}
\begin{corollary}\label{crl5}
    For  $-1  \leq B < A \leq 1$, let
     $$ 4 (A - B)^2 (4 + B) \leq 7 (A - B)^3 \leq (A - B)^2 (2 + 5 A^2 + 4 B - 17 A B + 14 B^2)  $$
     and
     $$  q_1= \frac{3 A - 7 B}{2}, \quad q_2= \frac{A^2 - 5 A B + 6 B^2}{2}.$$
     If $f\in \mathcal{C}[A,B]$ and $(q_1, q_2) \in \cup_{i=1}^3 D_i$, then
     $$ \vert T_{3,2}(f) \vert \leq \frac{1}{3456} (A - B)^2 ( 5 A^2 - 17 A B + 14 B^2 +36 ) ( A^3 +12 A - 12 B - 6 A^2 B + 11 A B^2 - 6 B^3).$$
     The estimates is sharp.
\end{corollary}
    When $A= 1- 2\alpha$ and $B=-1$, the conditions in Corollary \ref{crl4} and \ref{crl5} are true and $(q_1, q_2) \in D_3$  for $\alpha \in [0,1/7]$. Thus, we obtain the following bounds for the classes $\mathcal{S}^*(\alpha)$ and $\mathcal{C}(\alpha)$.
\begin{corollary}
    If $f\in \mathcal{S}^*(\alpha)$, then
    $$ \vert T_{3,2}(f) \vert \leq  \frac{4}{9} (1 - \alpha)^3 ( 16 \alpha^4 - 100 \alpha^3 + 268 \alpha^2 - 345 \alpha + 189) $$
     for $\alpha \in [0,1/7]$.
     The bound is sharp.
\end{corollary}
\begin{corollary}
      If $f\in \mathcal{C}(\alpha)$, then
    $$ \vert T_{3,2}(f) \vert \leq  \frac{1}{108} (1 - \alpha)^3 ( 20 \alpha^4  -124 \alpha^3 + 381 \alpha^2 - 576 \alpha  + 432 ) $$
    for $\alpha \in [0,1/7]$.
    The bound is sharp.
\end{corollary}
\begin{remark}
    In particular, when $\alpha=0$,  the following bounds for the classes $\mathcal{S}^*$ and $\mathcal{C}$ follow as special case proved in~\cite{AliW}.
\begin{enumerate}[(i)]
  \item If $f\in \mathcal{S}^*$, then $\vert T_{3,2}(f) \vert \leq 84$~\cite[Theorem 2.3]{AliW}.
  \item If $f\in \mathcal{C}$, then $\vert T_{3,2}(f) \vert \leq 4$~\cite[Theorem 2.11]{AliW}.
\end{enumerate}
\end{remark}
     In case of $\varphi(z) =((1+z)/(1-z))^{\beta}$, $0 < \beta \leq 1$, the classes $\mathcal{S}^*(\varphi)$ and $\mathcal{C}(\varphi)$ reduce to the class of strongly starlike functions of order $\beta$ and the class of strongly convex functions of order $\beta$, denoted by $\mathcal{SS}^*(\beta)$ and $\mathcal{CC}(\beta)$ respectively (see~\cite{Duren}).
\begin{corollary}
  If $f\in \mathcal{SS}^*(\beta)$, then
  $$  \vert T_{3,2}(f) \vert \leq \frac{4}{81} \beta^3 (160 + 742 \beta^2 + 799 \beta^4) $$
  for $\beta \in [3/4,1]$. The bound is sharp.
\end{corollary}
\begin{corollary}
   If $f \in \mathcal{CC}(\beta)$, then
   $$ \vert T_{3,2}(f) \vert \leq \frac{1}{324} \beta^3 (323 + 650 \beta^2 + 323 \beta^4)$$
   for $\beta \in [8/9,1]$. The bound is sharp.
\end{corollary}
  For $-1/2 < \lambda \leq 1 $ and $f\in \mathcal{A}$ such that $f$ is a locally univalent functions, Robertson~\cite{Robert} considered the class
   $$ \mathcal{F}(\lambda) = \bigg\{ f \in \mathcal{A} : \RE \bigg(1 + \frac{z f''(z)}{f'(z)} \bigg) > \frac{1}{2} - \lambda \bigg\}. $$
   Clearly, when $1/2 \leq \lambda\leq 1$, functions in $\mathcal{F}(\lambda)$ are close-to-convex~\cite{Kaplan}. For $ - 1/2 <  \lambda \leq 1/2$, the functions in $\mathcal{F}(\lambda)$ are convex. Vasudevarao et al.~\cite{Vasu} derived the sharp bound of $\vert T_{3,2}(f)\vert$ for $f\in \mathcal{F}(\lambda)$ when $1/2 \leq \lambda\leq 1$, that is the class of Ozaki close-to-convex functions. Consider
   $$ \varphi_\lambda(z) = \frac{1+ 2 \lambda z}{1-z},  \quad z \in \mathbb{D}.$$
   The function $\varphi_\lambda$ maps the unit disk onto the right half plane  for $- 1/2 <  \lambda \leq 1/2 $ such that $\RE \varphi_\lambda > (1/2- \lambda)$. Clearly, $\mathcal{C}(\varphi_\lambda) \subset \mathcal{F}(\lambda)$ for $\lambda \in (-1/2 ,1]$ and $\mathcal{C}(\varphi_\lambda) = \mathcal{F}(\lambda)$ when $\lambda \in (-1/2, 1/2]$.
   The Taylor's series expansion of $\varphi_\lambda$ gives $B_1 = B_2 = B_3 = (1 + 2 \lambda)$, which satisfy the condition~(\ref{26}) for $\lambda \in [5/14,1/2]$.
   Thus, from Theorem~\ref{thm6}, we obtain the following sharp bound of $\vert T_{3,2}(f)\vert$ for the class $\mathcal{F}(\lambda)$ when $5/14 \leq \lambda \leq 1/2$.
\begin{corollary}\label{crlr}
   If $f\in \mathcal{F}(\lambda)$, $5/14 \leq \lambda \leq 1/2$, then
   $$ \vert T_{3,2}(f) \vert \leq  \frac{1}{864} (1 + 2 \alpha)^3 (9 + 5 \alpha + 2 \alpha^2) (25 + 17 \alpha + 10 \alpha^2). $$
\end{corollary}
\begin{remark}
        Vasudevarao et al.~\cite[Theorem 4.3]{Vasu} proved the same bound as given in Corollary~\ref{crlr} for $1/2 \leq \lambda \leq 1$. Thus, Corollary~\ref{crlr} shows that the result is also true for $5/14 \leq \lambda \leq 1/2 $.
\end{remark}
%%%%%%%%%%%%%%%%%%%%%%%%%%%%%%%%%%%%%%%%%%%%%%%%%%%%%%%%%%%%%%%%%%%%%%%%%%%%%%%%%%%%%%%%%%%%%%%%%%%%%%%%%%%%%%%%%%%%%%%%%%
%%%%%%%%%%%%%%%%%%%%%%%%%%%%%%%%%%%%%%%%%%%%%%%%%%%%%%%%%%%%%%%%%%%%%%%%%%%%%%%%%%%%%%%%%%%%%%%%%%%%%%%%%%%%%%%%%%%%%%%%%%
\section*{Declarations}
\subsection*{Funding}
The work of the Surya Giri is supported by University Grant Commission, New-Delhi, India  under UGC-Ref. No. 1112/(CSIR-UGC NET JUNE 2019).
\subsection*{Conflict of interest}
	The authors declare that they have no conflict of interest.
\subsection*{Author Contribution}
    Each author contributed equally to the research and preparation of manuscript.
\subsection*{Data Availability} Not Applicable.
%%%%%%%%%%%%%%%%%%%%%%%%%%%%%%%%%%%%%%%%%%%%%%%%%%%%%%%%%%
\noindent
%%%%%%%%%%%%%%%%%%%%%%%%%%%%%%%%%%%%%%%%%%%%%%%%%%%%%%%%%%%%%%%%%%%%%%%%%%%%%%%%%%%%%%%%%%%%%%%%%%%


\begin{thebibliography}{99}

\bibitem{Logr} E. A. Adegani,  N. E. Cho and  M. Jafari, Logarithmic coefficients for univalent functions defined by subordination, Mathematics 7, no. 5 (2019), 408.

\bibitem{Ahuja} O.~P. Ahuja, K. Khatter\ and\ V. Ravichandran, Toeplitz determinants associated with Ma-Minda classes of starlike and convex functions, Iran. J. Sci. Technol. Trans. A Sci. {\bf 45} (2021), no.~6, 2021--2027.

\bibitem{AliW} M.~F. Ali, D.~K. Thomas\ and\ A. Vasudevarao, Toeplitz determinants whose elements are the coefficients of analytic and univalent functions, Bull. Aust. Math. Soc. {\bf 97} (2018), no.~2, 253--264.

\bibitem{AliFS} R. M. Ali, V. Ravichandran and N. Seenivasagan, Coefficient bounds for p-valent functions, Applied Mathematics and Computation 187.1 (2007): 35-46.

\bibitem{Allu} M.~F. Ali\ and\ A. Vasudevarao, On logarithmic coefficients of some close-to-convex functions, Proc. Amer. Math. Soc. {\bf 146} (2018), no.~3, 1131--1142.

\bibitem{ChoL} N.~E. Cho, B. Kowalczyk, O. S. Kwon, A. Lecko and Y. J. Sim, On the third logarithmic coefficient in some subclasses of close-to-convex functions, Rev. R. Acad. Cienc. Exactas F\'{\i}s. Nat. Ser. A Mat. RACSAM {\bf 114} (2020), no.~2, Paper No. 52, 14 pp.

\bibitem{Cudna} K. Cudna, O. S. Kwon, A. Lecko, Y. J. Sim, and B. \'{S}miarowska, The second and third-order Hermitian Toeplitz determinants for starlike and convex functions of order $\alpha$, Bol. Soc. Mat. Mex. (3) {\bf 26} (2020), no.~2, 361--375.

\bibitem{Duren}  P.~L. Duren, Univalent functions, Grundlehren der mathematischen Wissenschaften, 259, Springer, New York, 1983.

\bibitem{Efraim} I. Efraimidis, A generalization of Livingston's coefficient inequalities for functions with positive real part, J. Math. Anal. Appl. {\bf 435} (2016), no.~1, 369--379.

\bibitem{Giri}  S. Giri and S. S. Kumar, Hermitian-Toeplitz determinants for certain univalent functions, Anal. Math. Phys. ({\it{accepted}}).

\bibitem{Janow} W. Janowski, Some extremal problems for certain families of analytic functions. I, Ann. Polon. Math. {\bf 28} (1973), 297--326.

\bibitem{Kaplan} W. Kaplan, Close-to-convex schlicht functions, Michigan Math. J. {\bf 1} (1952), 169--185 (1953).

\bibitem{Kayumov} I.~R. Kayumov, On Brennan's conjecture for a special class of functions, Math. Notes {\bf 78} (2005), no.~3-4, 498--502; translated from Mat. Zametki {\bf 78} (2005), no. 4, 537--541.

\bibitem{Kowal} B. Kowalczyk\ and\ A. Lecko, Second Hankel determinant of logarithmic coefficients of convex and starlike functions, Bull. Aust. Math. Soc. {\bf 105} (2022), no.~3, 458--467.

\bibitem{Kowal2} B. Kowalczyk\ and\ A. Lecko, Second Hankel determinant of logarithmic coefficients of convex and starlike functions of order alpha, Bull. Malays. Math. Sci. Soc. {\bf 45} (2022), no.~2, 727--740.

\bibitem{Lecko} A. Lecko, Y.~J. Sim\ and\ B. \'{S}miarowska, The fourth-order Hermitian Toeplitz determinant for convex functions, Anal. Math. Phys. {\bf 10} (2020), no.~3, Paper No. 39, 11 pp.

\bibitem{MaMinda} W. Ma and D. Minda, A unified treatment of some special classes of univalent functions, Proceedings of the Conference on Complex Analysis, 1992. International Press Inc., 1992.

\bibitem{Milin} I. M. Milin, Univalent Functions and Orthonormal Systems, Izdat. “Nauka”, Moscow (1971) (in Russian); English transl., American Mathematical Society, Providence (1977).

\bibitem{Mundalia} M. Mundalia\ and\ S.~S. Kumar, Coefficient Problems for Certain Close-to-Convex Functions, Bull. Iranian Math. Soc. {\bf 49} (2023), no.~1, Paper No. 5.

\bibitem{Obra} M. Obradovi\'{c} and  N. Tuneski, Hermitian Toeplitz determinants for the class of univalent functions. Armenian Journal of Mathematics, 13(4), 1–10 (2021).

\bibitem{Prokh} D. V. Prokhorov\ and\ J. Szynal, Inverse coefficients for $(\alpha,\beta )$-convex functions, Ann. Univ. Mariae Curie-Sk\l odowska Sect. A {\bf 35} (1981), 125--143 (1984).

\bibitem{Robert} M.~S. Robertson, On the theory of univalent functions, Ann. of Math. (2) {\bf 37} (1936), no.~2, 374--408.

\bibitem{Thom2} D.~K. Thomas, On the coefficients of Bazilevi\v{c} functions with logarithmic growth, Indian J. Math. {\bf 57} (2015), no.~3, 403--418.

\bibitem{Thom}   D. K. Thomas, N. Tuneski, A. Vasudevarao, Univalent Functions: A Primer. Walter de Gruyter GmbH, Berlin (2018).

\bibitem{OTplz} O. Toeplitz, Zur Transformation der Scharen bilinearer Formen von unendlichvielen Ver\"{a}nderlichen. Mathematischphysikalis- che, Klasse, Nachr. der Kgl. Gessellschaft derWissenschaften zu G\"{o}ttingen, pp 110–115 (1907).

\bibitem{Vasu} A. Vasudevarao, A. Lecko\ and\ D.~K. Thomas, Hankel, Toeplitz, and Hermitian-Toeplitz determinants for certain close-to-convex functions, Mediterr. J. Math. {\bf 19} (2022), no.~1, Paper No. 22, 17 pp.

\bibitem{LHLIM} K. Ye\ and\ L.-H. Lim, Every matrix is a product of Toeplitz matrices, Found. Comput. Math. {\bf 16} (2016), no.~3, 577--598.

\end{thebibliography}
\end{document}